\documentclass[reqno,english]{amsart}

\usepackage{packages}
\usepackage{macros}

\newcommand*{\dreg}{\mathcal{G}_{n,\Delta}}

\newcommand*{\diam}{C_\mathrm{diam}} 
\newcommand*{\net}{\mathcal{N}}
\newcommand*{\netB}{\mathcal{N_{\mathcal B}}}
\newcommand*{\seeds}{\mathcal{S}}
\newcommand*{\vx}{\mathbf{x}}
\newcommand*{\hx}{\hat{\vx}}

\newcommand*{\tuples}{\Xi}

\def\Vol{{\rm Vol}}

\def\dist{{\rm dist}}
\def\geom{{\rm geom}}
\def\calL{{\mathcal{L}}}
\def\calD{{\mathcal{D}}}

\newtheorem*{theorem*}{Theorem}
\newtheorem*{problem}{Problem}

\title{Universal geometric non-embedding of random regular graphs}
\author[D.\ J.\ Altschuler \and K.\ Tikhomirov]{
        Dylan J. Altschuler \and 
        Konstantin Tikhomirov 
        }

\address{Dylan J. Altschuler, Department of Mathematical Sciences, Carnegie Mellon University.}
\address{Konstantin Tikhomirov, Department of Mathematical Sciences, Carnegie Mellon University.}
\begin{document}

\begin{abstract}
    Let $\Delta\geq 3$ be fixed, $n\geq n_\Delta$ be a large integer. It is a classical result that $\Delta$--regular expanders on $n$ vertices are not embeddable as geometric (distance) graphs into Euclidean space of dimension less than $c\log n$, for some universal constant $c$. We show that for typical $\Delta$-regular graphs, this obstruction is universal with respect to the choice of norm. More precisely,
    for a uniform random $\Delta$-regular graph $G$ on $n$ vertices, it holds with high probability: there is no normed space of dimension less than $c\log n$ which admits a geometric graph isomorphic to $G$.  The proof is based on a seeded multiscale $\varepsilon$--net argument. \\ 
\end{abstract}

\maketitle
\section{Introduction}

Let $X$ be a normed space with a norm $\|\cdot\|_X$, and let $V$ be a collection of $n$
distinct points in $X$. A graph $\Gamma=(V,E)$
with the edge set
$$
E:=\big\{\{v,w\}\subset V:\;\|v-w\|_X\leq 1\big\}
$$
is said to be a {\it geometric} (or a {\it distance}) graph in $X$\footnote{We note that there are varying definitions of distance graphs in the literature. The definition we adopt here is natural in the context of high-dimensional convex geometry.}. If, for a normed space $X$ and a given graph $G$, there does not exist a distance graph isomorphic to $G$, we say $G$ is \textit{geometrically non-embeddable} into $X$. 
Geometric graphs arise naturally in data science as
similarity networks in a feature space
where distinction between data points
is quantified by a norm. An aspect of particular interest is the construction of geometric graph
embeddings:
\begin{problem}
Given an $n$--vertex graph $G$ and a normed space $X$,
construct a distance graph in $X$ isomorphic to $G$, or show that $G$ is geometrically non-embeddable into $X$.
\end{problem}

This problem has been actively studied due to algorithmic applications, especially in the Euclidean setting. The minimum dimensions $d$ for which a graph $G$ admits a geometric embeddings into $\ell_2^d$ and $\ell_\infty^d$, respectively, are sometimes called the sphericity and cubicity of $G$. We refer to \cite[Section~2]{MoserPach} and \cite[Section~3]{BiluLinial}
for a comprehensive discussion of related results.
For Euclidean spaces of constant dimension greater than one,
even the corresponding decision problem
is known to be NP hard \cite{NPHard2,NPHardGtr2}. 
However, the problem has been efficiently solved in some special cases.
In particular, it was shown in \cite{FranklMaehara}
(see also \cite{RRS89} for an earlier result as well as
\cite{Average} for a generalization)
that any $n$--vertex graph of maximal degree $\Delta$ admits
a geometric embedding into a Euclidean space
of dimension $C\Delta^2\log n$, where $C>0$
is a universal constant,
and, moreover, that embedding can be constructed algorithmically.

In the other direction, it was shown in \cite{RRS89}
that any $n$--vertex
graph $G$ does not admit
geometric embeddings into Euclidean spaces
of dimension smaller than $\frac{\log \alpha(G)}{\log (2r(G)+1)}$,
where $\alpha(G)$ is the independence number of $G$, and
$r(G)$ is its radius. In particular, this implies that
with high probability, a random $\Delta$--regular ($\Delta\geq 3$) $n$--vertex
graph is not embeddable into Euclidean spaces of dimension smaller than
$\frac{c_\Delta\log n}{\log \log n}$ (see also \cite{RRSLower}).

More recent developments within the theory of metric embeddings
provide a functional-analytic viewpoint for proving geometric non-embedding of expander graphs.
Given a $\Delta$--regular graph $G$
on $\{1,2,\dots,n\}$, a parameter $p\in[1,\infty)$,
and a normed space $X$,
the {\it non-linear Poincar\'e constant} $\gamma(G,\|\cdot\|_X^p)$
is defined as the smallest positive number such that
for every choice of $n$ points $x_1,x_2,\dots,x_n$
in $X$,
\begin{equation}\label{eq:Poincare}
\frac{1}{n^2}\sum_{i,j=1}^n \|x_i-x_j\|_X^p
\leq \frac{\gamma(G,\|\cdot\|_X^p)}{\Delta n}\sum_{(i,j):\,i\sim j}
\|x_i-x_j\|_X^p,
\end{equation}
where we write $i\sim j$ for adjacent vertices $i,j$ of $G$.

Note that whenever $G$ admits a geometric embedding $\Gamma$ in $X$,
the parameter $\gamma(G,\|\cdot\|_X)$ provides an upper bound
on the expected distance (in $X$) between two randomly
chosen vertices of $\Gamma$. On the other hand,
a simple volumetric argument yields a {\it lower} bound
for the average distance of the form $\Omega((n/\Delta)^{1/\dim X})$
(see Section~\ref{Poincare} for details).
A combination of these two observations produces a natural
{\it necessary} condition for the existence of a geometric embedding for arbitrary $\Delta$--regular graph $G$:
\begin{equation}\label{eq: suf condition}
\gamma(G,\|\cdot\|_X)=\Omega((n/\Delta)^{1/\dim X}),
\end{equation}
where we use a standard asymptotic notation $\Omega(\cdot)$.

We recall that an $n$--vertex $\Delta$--regular graph $G$
is a {\it spectral expander} with spectral gap at least $\beta$,
if $1-\frac{1}{\Delta}\,\lambda_2(G)\geq \beta$,
where $\lambda_2(G)$ is the second eigenvalue
of the adjacency matrix of $G$. 
It is known that Euclidean spaces
satisfy the Poincar\'e inequality with a constant
depending only on the spectral gap of the graph and not on the space dimension \cite{Eskenazis2022}.
Together with \eqref{eq: suf condition}, this implies: 
\begin{proposition}[Corollary of the Poincar\'e
inequality in Euclidean spaces]\label{prop:P lp}
For every $\beta>0$
there is $c>0$ depending only on $\beta$
with the following property.
Let $\Delta\geq 3$, and for a large $n$
with $\Delta n$ even,
let $G_n$ be a $\Delta$--regular graph
on $n$ vertices with spectral gap at least $\beta$.
Then for every $d\leq c \log n$
there is no geometric embedding of $G_n$ into
a $d$--dimensional Euclidean space.
\end{proposition}
In view of \cite{Friedman2008},
for every constant $\Delta\geq 3$
a uniform random
$n$--vertex $\Delta$--regular graph has a spectral
gap bounded below by a universal constant with probability $1-o(1)$.
Therefore, the above proposition implies that
with high probability a uniform random
$\Delta$--regular graph on $[n]$ does not admit
geometric embeddings into Euclidean spaces of dimension
smaller than a constant multiple of $\log n$.

The statement of Proposition~\ref{prop:P lp}
generalizes to arbitrary normed spaces 
that enjoy ``dimension--independent'' Poincar\'e constants.
In particular, the result holds for normed spaces with a 
$1$--unconditional basis and non-trivial cotype \cite{MR3208067,MR1301394,ADTT2024}. 
However, it is known that for certain spaces,
such as $\ell_\infty^d$,
the parameter $\gamma(G,\|\cdot\|_X)$ grows
with the space dimension \cite{Eskenazis2022}.
The dimension--optimal general upper bound on $\gamma(G,\|\cdot\|_X)$,
independent of the geometry of the normed space $X$,
was obtained in \cite{Naor2021}
(see also the survey \cite{Eskenazis2022}):
$$
\gamma(G,\|\cdot\|_X)=O\Big(\frac{\log(\dim X+1)}{1-\Delta^{-1}\lambda_2(G)}\Big).
$$
Combined with \eqref{eq: suf condition},
this gives the following result:
\begin{proposition}[A corollary of \cite{Naor2021} and \eqref{eq: suf condition}]\label{prop:P general}
For every $\beta>0$
there is $c>0$ depending on $\beta$
with the following property.
Let $\Delta\geq 3$, and let $n$ be a large integer with
$\Delta n$ even.
Let $G_n$ be a $\Delta$--regular graph
on $n$ vertices with spectral gap at least $\beta$,
and let $X$ be a normed space
of dimension $d\leq \frac{c\log n}{\log \log \log n}$.
Then $G$ does not
admit a geometric embedding into $X$.
\end{proposition}
We refer to Section~\ref{Poincare} of this note for a discussion of the above statements.
We emphasize that these statements are essentially known, as they can be obtained by simple combinations
of existing results.

\bigskip

The main purpose of the present article
is to prove an analog of
Proposition~\ref{prop:P lp} for {\it arbitrary} normed spaces,
i.e remove the extra $\log\log\log n$ from the statement
of Proposition~\ref{prop:P general}.
Whereas Proposition~\ref{prop:P general} deals
with arbitrary deterministic expanders, we prove
the result for {\it random} regular graphs:

\begin{theorem}[Main result]\label{thm:main}
Fix $\Delta\geq 3$.
For every large $n$
with $\Delta n$ even,
let $G_n$ be a uniform random $\Delta$--regular graph
on $n$ vertices. Then with probability $1-o(1)$,
$G_n$ has the following property:
for every $d\leq c_{\text{\tiny\ref{thm:main}}} \log n$ and for every $d$--dimensional
normed space $X$,
there is no geometrical embedding of $G_n$ into $X$.
Here, $c_{\text{\tiny\ref{thm:main}}}>0$ is a universal constant.
\end{theorem}

Without any attempt to optimize, the value $.001$ suffices for $c_{\text{\tiny\ref{thm:main}}}$. 

\begin{remark}
While our proof naturally extends to the regime
$\Delta=n^{\ao{1}}$, we prefer to focus on the constant
degree in this note. When $\Delta\to\infty$, 
the optimal threshold for embeddability is a function of 
both $d$ and $\Delta$, and finding the 
function even in the Euclidean setting appears to be an open problem as of this writing.
\end{remark}

\begin{remark}
    There are related notions of embeddings considered in the literature, such as monotone embeddings \cite{BiluLinial, ordinal}. It is plausible the methods of this article may be of use in studying non-embeddability of typical regular graphs in these settings. 
\end{remark}

\bigskip

{\bf Acknowledgment.} The second named author is partially supported by NSF grant DMS 2331037.

\section{Non-linear Poincar\'e inequalities, and geometric embeddings}\label{Poincare}

In this section, we apply
the non-linear Poincar\'e inequalities
mentioned in the introduction, to prove Propositions~\ref{prop:P lp}
and~\ref{prop:P general}.
Note that applications
of the Poincar\'e inequalities to {\it bi-Lipschitz}
embeddings of graphs into normed spaces
is a well established research direction (see, in particular,
\cite[Section~6]{Eskenazis2022}).
The geometric embeddings we consider in this note
are conceptually quite different.
At the same time, we would like to emphasize that
the statements proved in this section are simple
corollaries of known facts, and are not the main results of this note. \\

In order to make use of the Poincar\'e inequalities,
we utilize the following:
\begin{lemma}
Let parameters $n,\Delta$, and $d$ satisfy
$n\geq 2^{d+1}(\Delta+1)$.
Further, let $\Gamma$ be a $\Delta$--regular
$n$--vertex geometric graph in a normed space $X$
of dimension $d$, and denote the vertices by $x_1,x_2,\dots,x_n$. Then
$$
\frac{1}{n^2}\sum_{i,j=1}^n \|x_i-x_j\|_X
> \frac{1}{4}\Big(\frac{n}{2(\Delta+1)}\Big)^{1/d}.
$$
\end{lemma}
\begin{proof}
Without loss of generality, $X=\RR^d$ endowed with a norm $\|\cdot\|_X$.
Define $B:=\frac{1}{2}\big\{x\in X:\;\|x\|_X\leq 1/2\big\}$,
and observe that, in view of $\Delta$--regularity
and the definition of a geometric graph, for every $i\leq n$
we have
$$
\big|\big\{j\neq i:\;(x_j+B)\cap(x_i+B)\neq\emptyset\big\}\big|
\leq \Delta.
$$
Fix any $t\geq 1$ and any index $i\leq n$.
By the above inequality,
the sum of indicator functions
$$
\sum_{j=1}^n {\bf 1}_{x_j+B}
$$
is at most $\Delta+1$ everywhere on $x_i+(t+1)\,B$.
Hence,
$$
|\{j\leq n:\;x_j\in x_i+t B\}|
\leq \frac{(\Delta+1)\Vol_d((t+1)B)}{\Vol_d(B)}
\leq(\Delta+1)(2t)^d.
$$
Thus, taking $t:=\frac{1}{2}\big(\frac{n}{2(\Delta+1)}\big)^{1/d}\geq 1$,
we get
$$
\|x_j-x_i\|_X> \frac{1}{2}\Big(\frac{n}{2(\Delta+1)}\Big)^{1/d}
$$
for at least $n/2$ indices $j\leq n$.
The result follows.
\end{proof}

As an immediate corollary of the above lemma,
we obtain the following quantitative version of \eqref{eq: suf condition}:
\begin{proposition}\label{prop:dist}
Let $G$ be a $\Delta$--regular graph on $n$
vertices, and $X$ be a $d$--dimensional normed space,
where the parameters $n,\Delta$, and $d$ satisfy
$n\geq 2^{d+1}(\Delta+1)$.
Assume further that $\gamma(G,\|\cdot\|_X)\leq
\frac{1}{2}\big(\frac{n}{2(\Delta+1)}\big)^{1/d}$.
Then $G$ does not admit a geometric embedding into $X$.
\end{proposition}

It was shown in \cite{Naor2021}
(see also survey \cite[Theorem~1.3]{Eskenazis2022})
that for any $\Delta$--regular $n$--vertex
expander graph $G_n$ and any $d$--dimensional
normed space $X$, the Poincar\'e constant
$\gamma(G_n,\|\cdot\|_X)$ in \eqref{eq:Poincare}
is bounded above by $c\log d$, where
$c>0$ may only depend on the spectral gap of $G_n$.
If $\Delta\geq 3$ and
if $G_n$ is a uniform random $\Delta$--regular
graph on $n$ vertices then $G_n$ is an expander with
high probability \cite{Friedman2008},
and hence $\gamma(G_n,\|\cdot\|_X)=O(\log d)$.
Combined with Proposition~\ref{prop:dist},
this implies $G_n$ is not embeddable into $X$ w.h.p
whenever
$$
\log d\leq c\Big(\frac{n}{2(\Delta+1)}\Big)^{1/d},
$$
where $c>0$ may only depend on $\Delta$.
Solving the last inequality for $d$
leads to Proposition~\ref{prop:P general}
from the introduction.

\bigskip

We would like to emphasize that
the above approach via Proposition~\ref{prop:dist}
cannot be used to verify the statement
of Theorem~\ref{thm:main}.
Perhaps the most natural example is the
space $\ell_\infty^d$. While this is briefly discussed in
\cite[Section~6]{Eskenazis2022}, let us briefly sketch
the basic idea, originating from \cite{Matousek1997}, for completeness.
Let $G_n$ be the uniform random $\Delta$--regular
graph on $n$, and let $d\leq \log n$.
Pick $d$ vertices $i_1,i_2,\dots,i_d$
of $G_n$ uniformly at random
and independently of the edge set of $G_n$,
and define a mapping from $[n]$ to $\ell_\infty^d$
as
$$
i\in[n]\longrightarrow (\dist_{G_n}(i,i_k))_{k=1}^d,
$$
where $\dist_{G_n}(\cdot)$ is the
usual graph distance on $G_n$.
It is immediate from the definition that for any pair of adjacent
vertices $i\sim j$ in $G_n$ we have
$$
\big\|(\dist_{G_n}(i,i_k))_{k=1}^d
-(\dist_{G_n}(j,i_k))_{k=1}^d\big\|_\infty\leq 1.
$$
On the other hand, it can be shown that
for two independently uniformly chosen vertices $i,j$, with a constant
probability
$$
\big\|(\dist_{G_n}(i,i_k))_{k=1}^d
-(\dist_{G_n}(j,i_k))_{k=1}^d\big\|_\infty=\Omega(\log d).
$$
Therefore, $\gamma(G_n,\|\cdot\|_\infty)=\Omega(\log d)$,
and hence the above method using Proposition~\ref{prop:dist} fails to produce the optimal
(logarithmic in $n$)
bound for non-embeddability of $G_n$ into an arbitrary normed space.

\section{Construction of the discretization scheme}

The goal of this section is to produce, for a given finite-dimensional normed space $X$, a discretization $\hx:X^n \to \net$, where $\net$ is a discrete subset of the Cartesian product $X^n$ having both relatively small cardinality and ``sufficient density.'' The construction presented below can be viewed as a  version of the classical net argument with advanced features provided through multiscaling and seeding.

For the remainder of this article, $\eps$, $c_{0}$, $c_{\text{\tiny\ref{thm:main}}}$, $C_{\mathrm{BM}}$ and $\diam$ are some universal constants. In particular, these constants do not depend on $X$, $\Delta$, or $n$. Without making effort to optimize, the following choices suffice:
\begin{align*}
    \eps &:= .1 \\
    \diam &:= 8 \\
    c_0 &:= .01 \\
    C_{\mathrm{BM}} &:= 10 \\
    c_{\text{\tiny\ref{thm:main}}} &:= \min \pa{ \frac{\eps}{2\log(320)},\, \frac{\eps \log(3/c_0)}{10},\, \frac {C_{\mathrm{BM}}} 2} \ge .001 \,.
\end{align*}
We will think of a normed space $X$ as the space $\RR^d$ endowed with a norm $\|\cdot\|_X$. The unit ball in $X$ will be denoted by $B_X$. More generally, a ball of radius $r$ centered at a point $y\in X$ will be denoted by $B_X(y,r)$.
Given a point $y\in X$ and a non-empty closed subset $A$ of $X$,
the $\|\cdot\|_X$--projection of $y$ onto $A$ is any point in $A$ with the minimal $\|\cdot\|_X$--distance to $y$.
Whenever projections of points onto subsets are not uniquely defined, a representative will be chosen arbitrarily. Finally, ``$n$ sufficiently large'' below means $n \ge n_0$ for some function $n_0 := n_0(\eps, \diam, c_0, c_{\text{\tiny\ref{thm:main}}}, \Delta)$. We emphasize that $n_0$ does not depend on the underlying normed space $X$. However, as $n_0$ \textit{may} depend on $\Delta$ (which is, for our purposes, a fixed constant), we may assume $n$ is arbitrarily large compared to $\Delta$.

\begin{definition}[$n$--tuples]
In what follows, $\vx$ will always denote an $n$-tuple of points $(x_1,\dots,x_n) \in (\RR^d)^n$. Occasionally, we will view $\vx$ as a multiset (i.e disregard the ordering).
\end{definition}

\begin{definition}[Sparse $n$--tuples]
    Let $\|\cdot\|_X$ be any norm in $\RR^d$.
    Say that a tuple of points $\vx$ is {\it $(1/2,\,\Delta)$-sparse} if, for all $i\in[n]$, $\|x_i - x_j\|_X\le 1/2$ for at most $\Delta$ values of $j \in [n]\setminus \cb{i}$.
\end{definition}

As will be made precise shortly, we limit our search for embeddings of graphs into $n$-tuples $\vx$ that are sparse and have small diameter. The diameter restriction is in view of the following classical fact:
\begin{lemma}[\cite{diameter}]\label{lemma:diameter}
    There exists a universal constant $\diam > 0$ such that for any $\Delta \ge 3$, with high probability as $n$ diverges, an $n$-vertex $\Delta$-regular random graph has diameter at most $\diam\log(n)/2$. 
\end{lemma}

\begin{definition}[Domain]
Assume that $X$ is a normed space.
Let domain $\tuples=\tuples(X)$ be given by
\begin{equation}\label{eq:domain}
    \tuples=\tuples(X) := \cb{\vx \in (B_{X}(0;\, \diam \log n ))^n\,:\, \vx \text{ is $(1/2,\,\Delta)$-sparse in $X$}}\,.
\end{equation}
Note that $\tuples$ implicitly depends on parameters $n,\Delta$.
\end{definition}

The discretization of $X$ will be based on a multiscale net-argument with
seeding.
Recall that an {\it $r$-net} on a set $A \subset \RR^{d}$ in $X=(\RR^{d},\|\cdot\|_X)$  is a collection of points $(y_i)_i$ so that for all $x \in A$, there exists $y_i$ with $\|x - y_i\|_X \le r$. We will always assume that the points $y_i$ are contained in $A$ for convenience. We introduce the following notation: 
\begin{definition}[Simple nets]
    Given $r>0$ and $y\in X$, define $\net(y, r; c_0r)$ as an (arbitrary) $c_0r$-net in $B_X(y; r)$ of minimum cardinality. Further, let $\net_0$ denote an arbitrary $1$-net in $B_X(0; \diam \log n)$ of minimum cardinality.
\end{definition}

\begin{definition}[Dyadic sequence]
For every integer $\ell$ define $r_\ell := 2^{\ell}$.
\end{definition}

\begin{definition}[Multiscale net]\label{def:mulitscale-net}
    Let $X$ be a normed space. We define the {\it multiscale net} $\net_M$ by:
    \[
        \net_M := \net_0 \cup \pa{ \bigcup_{y \in \net_0} \bigcup_{\ell = 0}^{\lceil\log_2 (2\diam \log n)\rceil} \net(y,\, r_\ell; c_0 r_\ell) } \,.
    \]
\end{definition}
In what follows, we will use points from $\net_M$ to approximate locations
of vertices in a geometric graph embedding into $X$.

\begin{definition}[Scale of separation]
    Given $\vx = (x_1,\dots,x_n) \in (\RR^{d})^n$,
    define the \textit{local scale of separation for $x_i$} by:
    \[
        \ell_i := \ell_i(\vx,X) = \min\cb{\ell \in \mathbb{Z}\,:\, \ba{\cb{j\in [n]\,:\, \|x_i-x_j\|_X \le r_\ell }} \ge n^{2\eps}}\,, \quad 
        1\leq i\leq n\,.
    \]
    Informally, for every $i$, the local scale of separation for $x_i$ is the logarithm of the radius of a ball in $X$ centered at $x_i$ containing about $n^{2\eps}$ points $x_j$. We further write $\ell(\vx)=\ell(\vx,X)$ for the tuple $(\ell_i)_{i=1}^n$.
\end{definition}

\begin{remark}\label{boundsonell}
Note that whenever $\vx$ is $(1/2,\,\Delta)$-sparse and $n$ is sufficiently large, we necessarily have $\ell_i\geq 0$ for all $i\leq n$.
Moreover, whenever $\vx\in (B_{X}(0,r))^n$ for some $r>0$, we must have $\ell_i\leq \lceil\log_2(2r)\rceil$ for all $i$.
In particular, for every $\vx=(x_i)_{i=1}^n\in\tuples$
    we have $\ell_i\in \{0,1,\dots, \lceil\log_2(2\diam \log n)\rceil\}$
    for all $i\leq n$.
\end{remark}

The following simple property says that,
assuming the appropriate choice of the constants involved,
for every $\vx \in \tuples$ a vast majority of pairs
of components $x_i,x_j$ are at large distance from each other.
\begin{lemma}[Lower-bound on scales of separation]\label{lemma:scale-lb}
For all $\vx \in \tuples$ and all $i \in [n]$,
\[
    \ba{\cb{j \in [n] \,:\, \|x_i - x_j\|_X \le 36}  } < n^{2\eps}\,.
\]
In particular, $r_{\ell_i} \geq 36$ for all $i \in [n]$. 
\end{lemma}

\begin{remark}
It is easy to verify from the proof below that the constant $36$ in the lower bound for $r_{\ell_i}$ can be replaced with arbitrary fixed number, by adjusting $c_d$ accordingly.    
\end{remark}

\begin{proof}
    As $\Delta$ is fixed, we may assume $n$ is large enough that $\Delta+1 \le n^{\eps/2}$. Additionally, $d \le c_d \log n$ by assumption.
    Assume for the sake of contradiction there is a $(1/2,\,\Delta)$-sparse tuple $\vx$ in $(B_{X}(0;\, \diam \log n ))^n$ and an index $i \in [n]$ such that
    \[
        \ba{\cb{j \in [n] \,:\, \|x_i - x_j\|_X \le 36}  } \ge n^{2\eps}\,.
    \]
    Note that any covering of $B_{X}(x_i, 36)$ using translate of $B_X(0,\, 1/4)$ must have cardinality at least $n^{\eps}/(\Delta+1)$. (By the pigeonhole principle, any smaller covering would have a ball with more than $\Delta+1$ points from $\vx$, violating the $(1/2,\,\Delta)$-sparse assumption on $\vx$). By the standard duality between packing and covering, this implies $B_{X}(x_i, 37)$ contains as a subset at least $n^{\eps}/(\Delta+1)$ disjoint copies of $B_X(0, 1/8)$. By volumetric considerations, we obtain:
    \[
        \frac{n^{\eps}}{\Delta+1}\, 8^{-d} \le 37^d\,.
    \]
    However, by our choices of parameters, $320^d \le e^{\eps \log(n)/2} \le n^{\eps}/(\Delta+1)$, providing the desired contradiction. Finally, note that by construction of $\ell_i$, we have $\ell_i \geq \log_2(36)$, so that $r_{\ell_i} \geq 36$.
\end{proof}

In the next proposition we define {\it seeds}. Seeds are
small-cardinality subsets of $\vx$ that play a key role in the discretization scheme: $\vx$ will be embedded into simple (local) nets centered at the seeds. 

\begin{proposition}[Existence of good seeds]\label{prop:seeds}
    Let $n$ be sufficiently large (independent of $X$). For any $n$--tuple $\vx$, and for each integer $\ell\geq 0$, there exists a multiset $\seeds_\ell := \seeds_\ell(\vx) \subset \vx$, with $|S_\ell| =\lfloor n^{1-\eps}\rfloor$ and the following property. For any component $x_i$ of $\vx$ with $\ell_i = \ell$, there exists $y \in \seeds_\ell$ with 
    \[
        \|y - x_i \|_X \le r_{\ell_i}\,.
    \]
\end{proposition}
\begin{proof}
    Fix $\vx$ and $\ell$. Let $\seeds_\ell$ be the collection of $\lfloor n^{1-\eps}\rfloor$ points selected independently uniformly at random from $\vx$ (viewed as a multiset). Let $z \sim \mathrm{unif}(\vx)$. For all $x_i \in \vx$, we have:   
    \begin{align*}
        \PP{\bigcap_{y \in S_\ell} \cb{\|y - x_i \|_X > r_\ell}} = \PP{\|z - x_i \|_X > r_\ell}^{\lfloor n^{1-\eps}\rfloor} \le \pa{1 - \frac{n^{2\eps}}{n}}^{\lfloor n^{1-\eps}\rfloor} &= e^{-\Theta(n^{\eps})} \,.
    \end{align*}
    Taking a union bound over the $n$ possible values of $i$, it easily follows that with high probability (in particular, positive probability), our random selection of $\seeds_\ell$ has the desired property. Thus, by the probabilistic method, there must deterministically be some satisfactory realization of $\seeds_\ell$.
\end{proof}

\begin{definition}[Discretization]
    We define a {\it discretization scheme} in $(\RR^d)^n$ as a mapping of every tuple $\vx \in \tuples(X)$ into $\hx := \hx(\vx)\in \net_M^n$ constructed as follows. For each $\ell \in \{0,1,\dots, \lceil\log_2(2\diam \log n)\rceil\}$, construct $S_\ell$, as defined in \cref{prop:seeds}. Let $\hat \seeds_{\ell}$ denote a multiset which is the $\|\cdot\|_{X}$-projection of $\seeds_\ell$ onto $\net_0$, and let $\hat \seeds = \bigcup_{\ell} \hat \seeds_\ell\subset \net_0$. Then, for each $i \in [n]$:
    \begin{enumerate}
        \item Let $\hat s_i$ be the $\|\cdot\|_{X}$-projection of $x_i$ onto $\hat \seeds_{\ell_i}$. 
        \item Let $\hat x_i$ be the $\|\cdot\|_{X}$-projection of $x_i$ onto $\net(\hat s_i, r_{\ell_i}; c_0 r_{\ell_i}) \subset \net_M$. 
    \end{enumerate}
    Define $\hx := (\hat x_1,\dots, \hat x_n) \in \net_M^{n}$. 
\end{definition}

\begin{remark}\label{distxitohatxi}
Note that, by the definition of $S_{\ell_i}$, the distance
from $x_i$ to $S_{\ell_i}$ is at most $r_{\ell_i}$.
Since $\net_0$ is a $1$--net, the distance from $\hat \seeds_{\ell_i}$
to $x_i$ is at most $r_{\ell_i}+1$. Thus,
$\|x_i-\hat s_i\|_X\leq r_{\ell_i}+1$, $i\leq n$.
In turn, this implies $\|x_i-\hat x_i\|_X
\leq c_0 r_{\ell_i}+1$ for all $i\leq n$.
\end{remark}

Note that the above construction 
depends on the underlying normed space $X$, and the parameters $n$ and $\Delta$. We collect some properties of this discretization.

\begin{lemma}\label{lemma:discretization} Let $n$ be sufficiently large (independently of $X$).
    \begin{enumerate}
        \item (Cardinality estimate)
        \[
            \ba{\cb{(\hx(\vx),\, \ell(\vx))\,:\, \vx \in \tuples }} \le \cE{ n \log n}\,.
        \]

        \item (Preservation of small distances)
        Let $\vx \in \tuples$ and $\hx := \hx(\vx)$. For all $(i,j) \in [n]^2$ with $\|x_i-x_j\|_X \le 2$, 
        \[
            \|\hat x_i - \hat x_j\|_X \le \frac{1}{6}\,r_{\ell_i} \,.
        \]

        \item (Preservation of local scales) Let $\vx \in \tuples$ and $\hx := \hx(\vx)$. For all $i \in [n]$, 
        \[
            \ba{\cb{j \in [n]\,:\, \|\hat x_i - \hat x_j\|_X \le \frac{1}{3}\,r_{\ell_{i}}}} \le n^{2\eps}\,.
        \]
    \end{enumerate}
\end{lemma}

\begin{proof}
    We prove the items in order. \\

    \textbf{Proof of claim 1}. By definition, $|\net_0|$ is the covering number of $B_X(0, \diam \log n)$ using translates of $B_X(0,1)$. 
    By elementary volumetric considerations, along with the standard duality between covering and packing, 
    \[
        |\net_0| \le \pa{3 \, \diam \log n }^d \le e^{ \log(n)^2}\,.
    \]
    Recalling the definition of $\net(y,r_\ell;c_0r_\ell)$, we similarly have:
    \[
        \max_{\substack{y \in \RR^d \\ \ell \ge 0}}|\net(y,r_\ell;c_0r_\ell)| \le \pa{3/c_0}^d\,.
    \]
    Recall further that for every
    $\ell\in\{0,1,\dots, \lceil\log_2(2\diam \log n)\rceil\}$,
    $\hat S_\ell$ is a submultiset of $\net_0$ of size $\lfloor n^{1-\eps}\rfloor$.
    Fix for a moment
    any multisets $S_{\ell}^*\subset \net_0$ of size $\lfloor
    n^{1-\eps}\rfloor$, $0\leq \ell\leq
    \lceil\log_2(2\diam \log n)\rceil$,
    arbitrary numbers $\ell_i^*\in \{0,1,\dots, \lceil\log_2(2\diam \log n)\rceil\}$,
    and arbitrary elements $s_i^*\in \hat S_{\ell_i}^*$, $i\leq n$, and consider the set
    $$
    \big\{\hx(\vx)\,:\, \vx \in \tuples,\,\ell_i(\vx)=\ell_i^*,\,
    \hat s_i(\vx)=s_i^*,\;i\leq n;\;\;\hat S_\ell=S_\ell^*,\,\ell\leq
    \lceil\log_2(2\diam \log n)\rceil\big\}.
    $$
    By definition, for every $\vx \in \tuples$ with $\ell_i(\vx)=\ell_i^*$
    and $\hat s_i(\vx)=s_i^*$, the $i$--th component of $\hx$
    is the projection of $x_i$ onto the net $\net(s_i^*, r_{\ell_i^*}; c_0 r_{\ell_i^*})$ of size at most $\pa{3/c_0}^d$. Thus, the cardinality of the above set is at most
    $$
    \big(\pa{3/c_0}^d\big)^n.
    $$
    Further, given $\vx \in \tuples$ with $\ell_i(\vx)=\ell_i^*$, $i\leq n$,
    and $\hat S_\ell=S_\ell^*$, $\ell\leq
    \lceil\log_2(2\diam \log n)\rceil$,
    there are at most $\big(n^{1-\eps}\big)^n$
    admissible values for the tuple $\big(\hat s_i(\vx)\big)_{i=1}^n$, and hence
    $$
    \big|\big\{\hx(\vx)\,:\, \vx \in \tuples,\,\ell_i(\vx)=\ell_i^*,\,
    \;\hat S_\ell=S_\ell^*,\,\ell\leq
    \lceil\log_2(2\diam \log n)\rceil\big\}\big|
    \leq \big(\pa{3/c_0}^d\big)^n\cdot \big(n^{1-\eps}\big)^n.
    $$
    The upper estimate on the cardinality of $\net_0$ allows for estimating
    the number of potential choices for sets $\hat S_\ell$, 
    leading to the bound
    \begin{align*}
    \big|\big\{\hx(\vx)\,:\, \vx \in \tuples,\,\ell_i(\vx)=\ell_i^*
    \big\}\big|
    &\leq \big(\pa{3/c_0}^d\big)^n\cdot \big(n^{1-\eps}\big)^n
    \cdot |\net_0|^{(1+\lceil\log_2(2\diam \log n)\rceil)\lfloor n^{1-\eps}\rfloor}\\
    &\leq \big(\pa{3/c_0}^d\big)^n\cdot \big(n^{1-\eps}\big)^n
    \cdot \big|e^{ \log(n)^2}\big|^{(1+\lceil\log_2(2\diam \log n)\rceil)\lfloor n^{1-\eps}\rfloor}.
    \end{align*}
    Finally, there are at most $\big(1+\lceil\log_2(2\diam \log n)\rceil\big)^n$ possible realizations of $\ell_i(\vx)$, $i\leq n$.
    Combining the bounds and using the assumptions on $d$, we get the claim. \\

    \textbf{Proof of claim 2.} By triangle inequality and the construction of $\ell_j$:
    \[
        \ba{\cb{i' \in [n]\,:\, \|x_j - x_{i'}\|_X \le \|x_i-x_j\|_X + r_{\ell_i}}} \ge n^{2\eps}\,.
    \]
    As we have assumed $\|x_i - x_j\|_X \le 2$, it holds by definition of $\ell_j$ that $r_{\ell_j} \le 2(r_{\ell_i} + 2)$. 
    Then, by construction of the discretization $\hx$ (see Remark~\ref{distxitohatxi}),
    \[
        \|\hat x_i - \hat x_j\|_X \le \|x_i - x_j\|_X + \|x_i - \hat x_i\|_X + \|x_j - \hat x_j\|_X  \le   4 + c_0 (r_{\ell_i} + r_{\ell_j}) \le  4 + 3c_0 \,r_{\ell_i} + 4c_0\,.
    \]
    For all $i$, we have that $4 \le r_{\ell_i}/9$ by \cref{lemma:scale-lb}. In addition, $c_0 = .01$ by assumption. Thus, the right-hand side of the above inequality is at most $r_{\ell_i}/6$, as desired. \\

    \textbf{Proof of claim 3.} 
    We prove the claim by contradiction. Assume that
    there is $i\in[n]$ such that
    \[
        \ba{\cb{j \in [n]\,:\, \|\hat x_i - \hat x_j\|_X \le \frac{1}{3}\,r_{\ell_{i}}}} > n^{2\eps}\,.
    \]
    Recall that, by the definition of $r_{\ell_i}$,
    \[
        \ba{\cb{j \in [n]\,:\, \|x_i - x_j\|_X \le r_{\ell_{i}}/2}} < n^{2\eps}\,.
    \]
    Therefore, there exists $x_j$ with both
    $\|x_i - x_j\| > r_{\ell_i}/2$ and $\|\hat x_i - \hat x_j\| \le r_{\ell_i}/3$.
    By triangle inequality, one of the following must hold:
    \[
        \|\hat x_i - x_i \| \ge r_{\ell_i}/12\,, \quad \text{or} \quad \|\hat x_j - x_j \| \ge r_{\ell_i}/12\,.
    \]
    By construction of $\hx$ (see Remark~\ref{distxitohatxi}), we have that $\|\hat x_i - x_i\| \le c_0 r_{\ell_i}+1$ and $\|\hat x_j - x_j\| \le c_0 r_{\ell_j}+1$. Since $c_0 = .01$ and in view of Lemma~\ref{lemma:scale-lb}, the first of the two options above is impossible. Then the second option must hold, which implies
    \begin{equation}\label{akjnfaskfjn}
    c_0 r_{\ell_j}+1 \ge r_{\ell_i}/12.
    \end{equation}
    On the other hand, the definition of $\ell_j$
    implies that $\ell_j\leq\lceil\log_2(\|x_i - x_j\|_X+r_{\ell_i})\rceil
    \leq \lceil\log_2(3r_{\ell_i}/2)\rceil$, i.e $r_{\ell_j}\leq 3r_{\ell_i}$.
    Combining the last assertion with \eqref{akjnfaskfjn}, we get
    $$
    3c_0r_{\ell_i}+1\geq r_{\ell_i}/12,
    $$
    which again contradicts Lemma~\ref{lemma:scale-lb}. This proves the claim.
\end{proof}

\section{Proof of Main Result}\label{sec:pf}

The proof of \cref{thm:main} proceeds in two steps. First, we show that for any particular $X$, the proportion of $d$-regular graphs which do not admit robust obstacles to embedding (i.e. are  ``close'' to being embeddable) is at most $\exp\{-cn \log n\}$ for some $c > 0$. Second, using known metric entropy estimates on the Banach-Mazur compactum, we show that we can discretize the space of all possible $X$ with a net having cardinality $\cE{\mathrm o(n\log n)}$ that is fine-grained enough to preserve ``robust obstacles'' to embeddability.

\subsection{Preliminaries} We first introduce some relevant notation. As before, let $\dreg$ denote the collection of simple undirected $\Delta$--regular graphs on vertices $\{1,2,\dots,n\}$.

Let $\mathcal{B}_d$ denote the set of all $d$-dimensional normed spaces, equipped with the so-called (multiplicative) Banach-Mazur distance. Recall that the (multiplicative) Banach-Mazur distance\footnote{It is $\log d_{\mathrm{BM}}$ which actually forms a metric; under this metric, it is a simple fact that $\mathcal{B}_d$ is compact.} is given by:
\[
    d_{\mathrm {BM}}(X,X') := \inf \cb{\|u\|\|u^{-1}\| }\,,
\]  
where the infimum is over isomorphisms from $X$ to $X'$ and $\|u\|$ is the operator norm. Given a normed space $X$ and an $n$--tuple $\vx$ in $X$, we will further use notation $\geom_{X}(\vx)$ for the geometric graph in $X$ generated by $\vx$, i.e the graph with vertex set $\vx$, where any two vertices are adjacent if and only if the $\|\cdot\|_X$--distance between them is at most one. Finally, we write $\geom_{X}(\vx)\cong G$ for a graph $G$ on $[n]$
if the mapping $i\to x_i$, $i\leq n$, is an isomorphism between $G$ and $\geom_{X}(\vx)$. \\

The following proposition is the main technical result of the paper; the proof is deferred. 
\begin{proposition}[Robust non-embedding]\label{prop:non-embedding}
    Let $X$ be a $d$--dimensional normed space, and let $n$ be sufficiently large.
    Let $B_{\text{BM}}(X;\,2)$ denote the set of $d$-dimensional normed spaces having Banach--Mazur distance at most two from $X$. For $G$ uniformly drawn from $\dreg$,
    \[
        \PP[\dreg]{~ \bigcup_{(X', \,\vx) \in B_{\mathrm{BM}}(X,\,2) \,\times \,\Xi(X)}~ \cb{\geom_{X'}(\vx) \cong G} } \le \cE{- \frac{1}{4}\,n  \log n}\,.
    \]
\end{proposition}

Next, we record a known metric entropy estimate on the Banach-Mazur compactum. 
\begin{lemma}[\cite{pisier}]\label{lemma:BM-entropy}
    There is some positive universal constant $C_{\text{BM}}$ and a $2$-net $\netB(d)$ in $\mathcal{B}_d$ such that:
    \[
        |\netB(d)| \le \cE{e^{d\,C_{\text{BM}}}}\,.
    \]
\end{lemma}
Opening up the proof of the lemma in \cite{pisier}, it is easily verified that setting $C_{\text{BM}} = 10$ suffices. Let us check that the above results readily imply the main theorem. 

\begin{proof}[Proof of \cref{thm:main}]
    Recalling $d \leq c_{\text{\tiny\ref{thm:main}}} \log n$ and letting $c_{\text{\tiny\ref{thm:main}}} \leq C_{\text{BM}}/2$, we have:
    \[
        |\netB(d)| \le \cE{n^{c_{\text{\tiny\ref{thm:main}}} C_{\text{BM}}}} \le e^{\sqrt{n}}\,.
    \]
    Let $G$ be distributed uniformly on $\dreg$. Define the high-probability event $\calD$ that the diameter of $G$ is not atypically huge,
    \[
        \calD := \cb{\mathrm{diam}(G) \le \diam \log(n)/2}\,.
    \]
    Applying \cref{lemma:diameter}, translation invariance of the metrics in $X$ and $X'$, and using the definition of $\netB$, 
    \begin{align*}
        &\PP[\dreg]{\bigcup_{d \le c_{\text{\tiny\ref{thm:main}}} \log n}~ \bigcup_{X \in \mathcal B_d}~\cb{\exists\,  \vx \in (\RR^d)^n\,:\, \geom_{X}(\vx) \cong G} } \\
         &=\PP[\dreg]{\bigcup_{d \le c_{\text{\tiny\ref{thm:main}}} \log n}~ \bigcup_{X \in \mathcal B_d}~\cb{\exists\,  \vx \in (\RR^d)^n\,:\, x_1=0,\;\geom_{X}(\vx) \cong G} } \\
        &= \ao{1} + \PP[\dreg]{\bigcup_{d \le c_{\text{\tiny\ref{thm:main}}} \log n}~ \bigcup_{X \in \mathcal B_d}~ \cb{\exists\,\vx \in (\RR^d)^n\,:\,x_1=0,\; \geom_{X}(\vx) \cong G} \,\bigg|\, \calD} \\
        &\le \ao{1} +\sum_{d \le c_{\text{\tiny\ref{thm:main}}} \log n} ~\sum_{X \in \mathcal \netB(d)} ~\PP[\dreg]{~ \bigcup_{X' \in B_{\mathrm{BM}}(X,\,2)}~ \cb{\exists\,\vx \in (\RR^d)^n\,:\,x_1=0,\; \geom_{X'}(\vx) \cong G} \,\bigg|\, \calD} \,.
    \end{align*}
    Recall the definition of $\Xi := \Xi(X)$ given in \cref{eq:domain}. If $\vx \in (\RR^d)^n \setminus \Xi(X)$, then either (1) there is a $\|\cdot\|_X$ ball of radius $1/2$, which is contained in a $\|\cdot\|_{X'}$ unit ball centered at one of the components of $\vx$, with more than $\Delta+1$ points of the $n$-tuple $\vx$. Or, (2) there are some $i$, $j$ with $\|x_i - x_j\|_X > \diam \log n$, and thus $\|x_i - x_j\|_{X'} > \diam \log(n)/2$. Either way, $\geom_{X'}(\vx) \ncong G$ (by $\Delta$-regularity considerations for the first case, and by diameter considerations for the second case). Thus, continuing the above computation, 
    \begin{align*}
        &\le \ao{1} +\sum_{d \le c_{\text{\tiny\ref{thm:main}}} \log n} ~\sum_{X \in \mathcal \netB(d)} ~\PP[\dreg]{~ \bigcup_{(X', \vx) \in B_{\mathrm{BM}}(X,\,2) \times \Xi(X)}~ \cb{\geom_{X'}(\vx) \cong G} \,\bigg|\, \calD} \\
        &\le \ao{1} + (c_{\text{\tiny\ref{thm:main}}} \log n)\,e^{\sqrt{n}} \,e^{-\frac{1}{4}\,n\log n} \\
        &= \ao{1}\,.
    \end{align*}
    From the first to the second line, we have invoked \cref{prop:non-embedding}. The theorem is established.
\end{proof}

\subsection{Robust non-embedding}
We use the following classical enumeration result.

\begin{lemma}[\cite{enumerate}]\label{lemma:enumerate}
    Let $\Delta \le n^{\ao{1}}$. Then:
    \[
        |\dreg| = \mo{1}\,e^{1 - \frac{\Delta^2}{4}} \frac{(\Delta n)!}{(\Delta n /2 )!\,2^{\Delta n /2}(\Delta !)^n}
        \,.
    \]
\end{lemma}
In particular, the following consequence will be utilized later:
\begin{corollary}\label{cor:contiguity}
    Let $\Delta n /2$ be an integer, and let $G$ be drawn uniformly from graphs on $[n]$ with $\Delta n / 2$ edges, denoted as $G \sim G(n,\Delta n / 2)$. If $\Delta = n^{\ao{1}}$,
    \[
        \PP{G \in \dreg} = e^{-\ao{n \Delta \log n}}\,.
    \]
\end{corollary}
As $\Delta$ is assumed constant, \cref{cor:contiguity} roughly asserts that the change-of-measure between the uniform $\Delta$-regular model and the more tractable $G(n,m)$ Erd\"os-R\'enyi model is negligible with respect to the scale of the claimed probability estimate in \cref{prop:non-embedding}.
\begin{proof}[Proof of \cref{cor:contiguity}] We prove the result instead for $G$ drawn from the ``$G(n,p)$'' Erd\"os-R\'enyi model, $G(n,\Delta /n)$. The result then immediately extends to $G(n,\Delta n/2)$. 
By definition of the Erd\"os-R\'enyi model, we have:
\[
    \PP{G \in \dreg} = |\dreg| \pa{\frac{\Delta}{n}}^{\Delta n / 2}\pa{1 - \frac{\Delta}{n}}^{\frac n 2 (n - 1 - \Delta)} = |\dreg| \,\cE{-\frac{\Delta}{2}\,n\log n - \aO{\Delta n\log \Delta}} \,.
\]
Next, applying Stirling's formula to the conclusion of \cref{lemma:enumerate}, we obtain:
\begin{align*}
     |\dreg| &= \cE{\pm\aO{\Delta n \log \Delta }} \pa{\frac{\Delta n}{e}}^{\Delta n / 2} = \cE{\pm\aO{\Delta n \log \Delta }} \,\cE{\frac{\Delta}{2}\,n\log n }\,.
\end{align*}
Combining the previous two equations yields the result:
\[
        \PP{G \in \dreg} = \cE{-\aO{n \Delta \log \Delta}} = \cE{-\ao{n \Delta \log n}}\,.
    \]
\end{proof}

We are ready to prove the main result of this section. 
\begin{proof}[Proof of \cref{prop:non-embedding}]
    For now on, we fix any normed space $X=(\RR^d,\|\cdot\|_X)$.
    For $\vx \in \tuples(X)$, let $\calL(\vx,X)$ be the set of ``long distances'' for $\vx$, defined as:
    \[
        \calL(\vx,X) := \cb{(i,j)\,:\,\|\hat x_i - \hat x_j\|_X > \frac{1}{3}\,r_{\ell_i}}
    \]  
    (where the discretization $\hx$ of $\vx$ is also constructed in the space $X$).
    Observe that $\calL(\vx,X)$ depends on $\vx$ only through $\hx$ and $\ell(\vx)$. Thus, the first assertion of \cref{lemma:discretization} yields:
    \begin{equation}\label{eq:cardinality-long}
        \ba{\cb{\calL(\vx,X) \,:\, \vx \in \tuples(X)}} \le  \ba{\cb{(\hx,\ell(\vx)) \,:\, \vx \in \tuples(X)}} \le \cE{n \log n} \,.
    \end{equation}
    Fix for a moment any $\vx\in\tuples(X)$.
    We claim that if the edge-set $E(G)$ of a graph $G$ on $[n]$ has a non-empty intersection with $\calL(\vx,X)$, then for any $X' \in B_{\mathrm{BM}}(X,2)$,
    \[
        \geom_{X'}(\vx) \ncong G\,.
    \]
    Indeed, assume for contradiction there exists a pair $(i,j)\in
    \calL(\vx,X)$ with $\{i,j\}\in E(G)$, and $\geom_{X'}(\vx) \cong G$. By definition of a geometric graph, $\|x_i -x_j\|_{X'} \le 1$. As $X' \in B_{\mathrm{BM}}(X,2)$, it follows that $\|x_i - x_j\|_{X} \le 2$. By the second assertion of \cref{lemma:discretization}, 
    \[
        \|\hat x_i -\hat x_j\|_{X} \le \frac{1}{6}r_{\ell_i}\,.
    \]
    This implies that $(i,j) \not \in \calL(\vx,X)$, which supplies the desired contradiction. In what follows, by a slight abuse of notation
    we write $E(G)\cap \calL( \vx,X)=\emptyset$ whenever there is
    no pair $(i,j)\in \calL( \vx,X)$ with $\{i,j\}\in E(G)$.

\medskip
    
    In view of the proven claim, letting $E(G)$ denote the edge-set of an {\it arbitrary} random graph $G$ on $[n]$, we obtain:
    \[
        \PP{\exists\, (X',\vx) \in B_{\text{BM}}(X;\,2) \times \tuples(X) \,:\, \geom_{X'}(\vx) \cong G} \le \PP{\exists\, \vx \in  \tuples(X) \,:\, E(G)\cap \calL( \vx,X)=\emptyset}\,.
    \]
    In the remainder of this proof, let $\mathbb P_{\dreg}$ and $\mathbb P_{G(n,\Delta n / 2)}$ respectively denote the uniform measure on $\dreg$ and on the set of graphs on $[n]$ with $\Delta n /2$ edges. 
    We compute:
    \begin{align*}
        \PP[ \dreg]{\exists\, \vx \in  \tuples(X) \,:\, E(G)\cap \calL( \vx,X)=\emptyset} &= \PP[G(n,\Delta n / 2)]{\exists\, \vx \in  \tuples(X) \,:\, E(G)\cap \calL( \vx,X)=\emptyset\,\big|\,G \in \dreg} \\
        &\le \frac{\PP[G(n,\Delta n / 2)]{\exists\, \vx \in  \tuples(X) \,:\, E(G)\cap \calL( \vx,X)=\emptyset}}{\PP[G(n,\Delta n / 2)]{G \in \dreg}}\,.
    \end{align*}
    In the last line, the probability in the denominator is lower-bounded via \cref{cor:contiguity}. For the numerator, recalling \eqref{eq:cardinality-long}, a union bound yields:
    \begin{align*}
        \PP[G(n,\Delta n / 2)]{\exists\, \vx \in  \tuples(X) \,:\, E(G)\cap \calL( \vx,X)=\emptyset} &\le e^{n\log n} \max_{\vx \in \tuples(X)}\, \PP[G(n,\Delta n / 2)]{ E(G)\cap \calL( \vx,X)=\emptyset}\,.
    \end{align*}
    We claim that for all $\vx \in \tuples(X)$, \begin{equation}\label{eq:erdos-renyi}
        \PP[G(n,\Delta n / 2)]{ E(G)\cap \calL( \vx,X)=\emptyset} \le \cE{-\frac{(1-\eps)}{2}\,n\,\Delta\log(n) + \aO{n\Delta}}\,.
    \end{equation}
    Indeed, note that $|E(G) \cap \calL(\vx,X)|$ is distributed as a hypergeometric random variable with $n(n-1)/2$ total population, $|\calL(\vx,X)|$ population marked as ``success,'' and $n\Delta/2$ trials. Then:
    \begin{align*}
        \PP[G(n,\Delta n / 2)]{ E(G)\cap \calL( \vx,X)=\emptyset} = \frac{
        \binom{\binom{n}{2} - |\calL(\vx,X)|}{n\Delta/2}}{\binom{\binom{n}{2}}{n\Delta/2}}, 
    \end{align*}
    where, again by a slight abuse of notation, we think of $|\calL(\vx,X)|$ as the number of unordered pairs $\{i,j\}$, $i\neq j$, with $(i,j)$ or $(j,i)$ in $\calL(\vx,X)$.
    For any $\vx \in \tuples(X)$, we have by the third assertion of \cref{lemma:discretization} that $|\binom{n}{2}-|\calL(\vx,X)|| \le n^{1+2\eps}$. Applying the elementary inequality $(m/k)^k \le \binom{m}{k} \le (me/k)^k$ yields
    \begin{align*}
        \PP[G(n,\Delta n / 2)]{ E(G)\cap \calL( \vx,X)=\emptyset} &\le \pa{\frac{e n^{1+\eps}}{n(n-1)/2}}^{n\Delta/2} 
        = \cE{-\frac{1-\eps}{2}\,n\,\Delta \log n + \aO{n\Delta}} \,.
    \end{align*}
    This establishes \eqref{eq:erdos-renyi}. Combining \eqref{eq:erdos-renyi} with the computations preceding it yields the desired result: 
    \begin{align*}
        \PP[ \dreg]{\exists\, (X',\vx) \in B_{\text{BM}}(X;\,2) \times \tuples(X) \,:\, \geom_{X'}(\vx) \cong G}  &\le e^{n\log n}\,e^{-\frac{(1-\eps)}{2}\,n\,\Delta\log(n) + \ao{n\log n}} \\
        &\le e^{-\frac{1}{4}\,n\log n }\,.
    \end{align*}
\end{proof}

\bibliography{sphericity}

\begin{thebibliography}{10}

\bibitem{ADTT2024}
{\sc Altschuler, D.~J., Dodos, P., Tikhomirov, K., and Tyros, K.}
\newblock A combinatorial approach to nonlinear spectral gaps.
\newblock {\em arXiv preprint\/} (2024).

\bibitem{BiluLinial}
{\sc Bilu, Y., and Linial, N.}
\newblock Monotone maps, sphericity and bounded second eigenvalue.
\newblock {\em J. Combin. Theory Ser. B 95}, 2 (2005), 283--299.

\bibitem{diameter}
{\sc Bollob\'{a}s, B., and Fernandez de~la Vega, W.}
\newblock The diameter of random regular graphs.
\newblock {\em Combinatorica 2}, 2 (1982), 125--134.

\bibitem{NPHard2}
{\sc Breu, H., and Kirkpatrick, D.~G.}
\newblock Unit disk graph recognition is {NP}-hard.
\newblock {\em Comput. Geom. 9}, 1-2 (1998), 3--24.

\bibitem{ordinal}
{\sc B\u{a}doiu, M., Demaine, E.~D., Hajiaghayi, M.~T., Sidiropoulos, A., and Zadimoghaddam, M.}
\newblock Ordinal embedding: approximation algorithms and dimensionality reduction.
\newblock In {\em Approximation, randomization and combinatorial optimization}, vol.~5171 of {\em Lecture Notes in Comput. Sci.} Springer, Berlin, 2008, pp.~21--34.

\bibitem{Eskenazis2022}
{\sc Eskenazis, A.}
\newblock Average distortion embeddings, nonlinear spectral gaps, and a metric {J}ohn theorem [{\it after} {A}ssaf {N}aor].
\newblock {\em Ast\'{e}risque}, 438, S\'{e}minaire Bourbaki. Vol. 2021/2022. Expos\'{e}s 1181--1196 (2022), Exp. No. 1188, 295--333.

\bibitem{FranklMaehara}
{\sc Frankl, P., and Maehara, H.}
\newblock The {J}ohnson-{L}indenstrauss lemma and the sphericity of some graphs.
\newblock {\em J. Combin. Theory Ser. B 44}, 3 (1988), 355--362.

\bibitem{Friedman2008}
{\sc Friedman, J.}
\newblock A proof of {A}lon's second eigenvalue conjecture and related problems.
\newblock {\em Mem. Amer. Math. Soc. 195}, 910 (2008), viii+100.

\bibitem{NPHardGtr2}
{\sc Kang, R.~J., and M{\"u}ller, T.}
\newblock Sphere and dot product representations of graphs.
\newblock In {\em Proceedings of the twenty-seventh annual symposium on Computational geometry\/} (2011), pp.~308--314.

\bibitem{Matousek1997}
{\sc Matou\v{s}ek, J.}
\newblock On embedding expanders into {$l_p$} spaces.
\newblock {\em Israel J. Math. 102\/} (1997), 189--197.

\bibitem{enumerate}
{\sc McKay, B.~D., and Wormald, N.~C.}
\newblock Asymptotic enumeration by degree sequence of graphs with degrees {$o(n^{1/2})$}.
\newblock {\em Combinatorica 11}, 4 (1991), 369--382.

\bibitem{MoserPach}
{\sc Moser, W., and Pach, J.}
\newblock Recent developments in combinatorial geometry.
\newblock In {\em New trends in discrete and computational geometry}, vol.~10 of {\em Algorithms Combin.} Springer, Berlin, 1993, pp.~281--302.

\bibitem{MR3208067}
{\sc Naor, A.}
\newblock Comparison of metric spectral gaps.
\newblock {\em Anal. Geom. Metr. Spaces 2}, 1 (2014), 1--52.

\bibitem{Naor2021}
{\sc Naor, A.}
\newblock An average {J}ohn theorem.
\newblock {\em Geom. Topol. 25}, 4 (2021), 1631--1717.

\bibitem{MR1301394}
{\sc Odell, E., and Schlumprecht, T.}
\newblock The distortion problem.
\newblock {\em Acta Math. 173}, 2 (1994), 259--281.

\bibitem{pisier}
{\sc Pisier, G.}
\newblock On the metric entropy of the {B}anach-{M}azur compactum.
\newblock {\em Mathematika 61}, 1 (2015), 179--198.

\bibitem{RRSLower}
{\sc Reiterman, J., R\"{o}dl, V., and \v{S}i\v{n}ajov\'{a}, E.}
\newblock Embeddings of graphs in {E}uclidean spaces.
\newblock {\em Discrete Comput. Geom. 4}, 4 (1989), 349--364.

\bibitem{RRS89}
{\sc Reiterman, J., R\"{o}dl, V., and \v{S}i\v{n}ajov\'{a}, E.}
\newblock Geometrical embeddings of graphs.
\newblock {\em Discrete Math. 74}, 3 (1989), 291--319.

\bibitem{Average}
{\sc Reiterman, J., R\"{o}dl, V., and \v{S}i\v{n}ajov\'{a}, E.}
\newblock On embedding of graphs into {E}uclidean spaces of small dimension.
\newblock {\em J. Combin. Theory Ser. B 56}, 1 (1992), 1--8.

\end{thebibliography}
\bibliographystyle{acm}

\end{document}